\documentclass[11pt,a4paper,oneside,leqno]{amsart}
\usepackage[utf8]{inputenc}
\usepackage[english]{babel}
\usepackage{amsmath}
\usepackage{amsfonts}
\usepackage{amssymb}
\usepackage{amsthm}
\usepackage{color}
\usepackage{graphicx}

\usepackage{enumitem}
\reversemarginpar

\newcommand{\marginp}[1]{}

\newcommand{\comm}[1]{}

\newtheorem{thm}{Theorem}[section]
\newtheorem{lemma}[thm]{Lemma}
\newtheorem{prop}[thm]{Proposition}
\newtheorem{cor}[thm]{Corollary}

\newtheorem*{thmA}{Theorem A}
\newtheorem*{thmB}{Theorem B}

\newtheorem*{thms}{Theorem}

\theoremstyle{definition}

\newtheorem{defi}[thm]{Definition}
\newtheorem{rem}[thm]{Remark}

\newtheorem*{notas}{Note}
\newtheorem*{defis}{Definition}

\newcommand{\N}{\mathbb{N}}
\newcommand{\Z}{\mathbb{Z}}
\newcommand{\R}{\mathbb{R}}
\newcommand{\C}{\mathbb{C}}
\newcommand{\cC}{\mathcal{C}}
\newcommand{\D}{\mathbb{D}}
\newcommand{\cH}{\mathcal{H}}

\newcommand{\T}{\mathbb{T}}

\newcommand{\wT}{\widetilde{T}}
\newcommand{\wH}{\widetilde{H}}
\newcommand{\hh}{\mathbf{h}}
\newcommand{\bfg}{\mathbf{g}}

\newcommand{\al}{\alpha}
\newcommand{\f}{\alpha}

\newcommand{\abs}[1]{| #1 |}
\newcommand{\Abs}[1]{\left| #1 \right|}
\newcommand{\norm}[1]{\| #1 \|}

\newcommand{\Norm}[1]{{\vert\kern-0.25ex\vert\kern-0.25ex\vert #1 \vert\kern-0.25ex\vert\kern-0.25ex\vert}}

\newcommand{\tn}[1]{\textnormal{#1}}
\newcommand{\wt}[1]{\widetilde{#1}}
\newcommand{\ol}[1]{\overline{#1}}
\renewcommand{\ss}{\subset}

\newcommand{\be}{\beta}
\newcommand{\ga}{\gamma}
\newcommand{\om}{\omega}
\newcommand{\la}{\lambda}
\newcommand{\ep}{\varepsilon}

\newcommand{\vp}{\varphi}

\renewcommand{\epsilon}{\varepsilon}
\renewcommand{\Re}{\operatorname{Re}}

\makeatletter
\newcommand*\bigcdot{\mathpalette\bigcdot@{.5}}
\newcommand*\bigcdot@[2]{\mathbin{\vcenter{\hbox{\scalebox{#2}{$\m@th#1\bullet$}}}}}
\makeatother

\newcommand{\gs}{good weight}

\date{\today}
\title{Operator inequalities implying similarity to a contraction}

\author{Glenier Bello-Burguet}
\address{Glenier Bello-Burguet\newline
Instituto de Ciencias Matem\'aticas (CSIC-UAM-UC3M-UCM), \newline
C/ Nicol\'as Cabrera, n${}^o$ 13-15 Campus de Cantoblanco, UAM, 28049 Madrid, Spain}

\email{gl.bello@icmat.es}
\author{Dmitry Yakubovich}
\address{D. V. Yakubovich\newline
Departamento de Matem\'aticas,\newline
Universidad Aut\'onoma de Madrid,\newline
Cantoblanco, 28049 Madrid, Spain\newline
and Instituto de Ciencias Matem\'aticas (CSIC-UAM-UC3M-UCM)}
\email {dmitry.yakubovich@uam.es}
\thanks{First author acknowledge the Grant Severo-Ochoa La Caixa for undergraduate studies.
Both authors are partially supported by Plan Nacional  I+D grant no. MTM2015-66157-C2-1-P.
The authors also acknowledge financial support from the Spanish Ministry of Economy and Competitiveness,
through the ``Severo Ochoa Programme for Centres of Excellence in R$\&$D'' (SEV-2015-0554).
}

\begin{document}

\begin{abstract}
Let $T$ be a bounded linear operator on a Hilbert space $H$ such that
\[
\alpha[T^*,T]:=\sum_{n=0}^\infty \alpha_n T^{*n}T^n\ge 0.
\]
where $\alpha(t)=\sum_{n=0}^\infty \alpha_n t^n$
is a suitable analytic function in the unit disc $\D$ with real coefficients.
We prove that if
$\alpha(t) = (1-t) \tilde{\alpha} (t)$, where
$\tilde{\alpha}$ has no roots in $[0,1]$, then $T$ is similar to a contraction.

Operators of this
type have been investigated by Agler, M\"uller, Olofsson, Pott and others, however,
we treat cases where their techniques do not
apply.
%

We write down an explicit Nagy-Foias type model
of an operator in this class and discuss
its usual consequences (completeness of eigenfunctions, similarity to a normal operator, etc.).
We also show that the limits of $\|T^nh\|$ as $n\to\infty$, $h\in H$, do  not exist
in general, but do exist if an additional assumption on $\al$ is imposed.

Our approach is based on a factorization lemma for certain weighted $\ell^1$
Banach algebras.
\end{abstract}

\maketitle


\section{Introduction}


Let $H$ be a separable complex Hilbert space and denote by $L(H)$
the set of bounded linear operators on $H$. Let $T \in L(H)$ and let
\[
\f(t)=\sum_{n=0}^\infty \f_n t^n
\]
be an analytic function on the open unit disc $\D=\{z: |z|<1\}$ such
that
$\f_n$ are real and
$\sum_n |\f_n|\|T^n\|^2<\infty$.
Then we define the so-called hereditary calculus
\[
\f[T^*,T] := \sum_{n=0}^{\infty} \f_n T^{*n} T^n.
\]
This work is devoted to the study of operators  $T \in L(H)$ that satisfy an operator inequality
\begin{equation}
\label{eq. f geq 0}
\f[T^*,T] \geq 0.
\end{equation}
Notice that for $\f(t) = 1-t$, this is just the class of all contractions on $H$.

The study of operator inequalities of this type was originated in the work
by Agler~\cite{Agl82}, where he studied more general inequalities
of the form $\sum_{j,k} \f_{jk} T^{*j} T^k \ge 0$.
Suppose that
$k(w,z)=1/\big(\sum_{j,k} \f_{jk} w^j z^k\big)$ is analytic in $\D\times \D$ and
$k(\bar w,z)$ is a reproducing kernel that defines a functional Hilbert space $\cH_k(\D)$
of functions on $\D$.
Let $M$ be the operator $Mu(z)=zu(z)$, acting on $\cH_k(\D)$, and assume that it is bounded.
Agler's main result in~\cite{Agl82} asserts that
an operator $T$ whose spectrum $\sigma(T)$ is contained in $\D$
satisfies the above inequality if and only if it is unitarily equivalent to the operator
$\bigoplus_{j=1}^\infty M^*$ restricted to an invariant subspace.
This is what Agler calls a coanalytic model of $T$.

The condition $\sigma(T)\subset \D$ is too restrictive; it has been shown in subsequent
papers that in many cases it can be replaced by $\sigma(T)\subset \overline{\D}$.
Then, in general, instead of the operator
$\bigoplus_{j=1}^\infty M^*$, the above coanalytic model involves a direct sum
$\bigoplus_{j=1}^\infty M^*\oplus U$, where $U$ is a unitary.
For instance, in \cite{Agl85}, Agler proved that an operator $T$
is a hypercontraction of order $n$ (that is, it satisfies \eqref{eq. f geq 0}  for
$\f(t) = (1-t)^k$, $k = 1, \ldots , n$)
iff it can be extended
to an operator $\bigoplus_{j=1}^\infty M^*\oplus U$, where
$M$ acts on the Bergman space, whose reproducing kernel is $1/(1-\bar wz)^n$.
If $T$ is of class $C_{0\, \bigcdot}$ (that is, $T^n\to 0$ strongly),
the above unitary summand is absent.


The case where $\f$ is a polynomial, say $\f=p$, was studied further by
M\"uller in \cite{Mul88}. He considers the class $C(p)$ of
operators $T \in L(H)$ such that \eqref{eq. f geq 0} is satisfied
and proves that $T$ has a coanalytic model whenever
$p(1)=0$, $1/p(t)$ is analytic in $\D$ and $1/p(\bar w z)$ is a reproducing kernel.
Notice that the last condition is equivalent to the fact that
all Taylor coefficients of $1/p(t)$ at the origin are positive.

In \cite{Olo15}, Olofsson deals with a more general setting, when
$\al$ is not a polynomial. Suppose an analytic function $\al(t)$ on $\D$ satisfies
$\al\ne 0$ in $\D$ and $1/\al$ has positive Taylor coefficients at the origin.
Olofsson studies contractions $T$ on $H$
that satisfy $\al[rT^*,rT]\ge 0$ for all $r$, $0\le r<1$
(he imposes some more assumptions on $\al$).
He obtains the coanalytic model for this class of operators.

%

Certain types of operator inequalities like \eqref{eq. f geq 0}
have also been studied for commuting tuples of operators in
\cite{AEM02}, \cite{Pott99}, \cite{BS06} and other papers. Pott in
\cite{Pott99} considered \emph{positive regular polynomials}.
These are polynomials of several complex variables with
non-negative coefficients such that the constant term is $0$ and
the coefficients of the linear terms are positive. Given such
polynomial $p$, Pott constructed a dilation model for commuting
tuples of operators satisfying the positivity conditions
$(1-p)^k[T^*,T] \geq 0$ for $1 \leq k \leq m$ (see Theorem 3.8 in
\cite{Pott99}). In \cite{BS06}, Bhattacharyya and Sarkar define
the characteristic function $\theta_T$ for this class of tuples
and construct a functional model in the pure case (see Theorem 4.2
in \cite{BS06}).

A general framework of Agler's theory in case of
operator inequalities for commuting tuples of operators has been
given in~\cite{AEM02} and further generalized in~\cite{AE03}; the latter work treats
general analytic models, which involve multi-dimensional analogues of operators
$\bigoplus_{j=1}^\infty M^*\oplus U$ attached to a domain in $\C^n$.

\medskip

Here we restrict ourselves to a single operator, but
for this case, we can deal with
a large class of operator inequalities,
\marginp{!!! Mira, por favor,esta frase}
for which the original Agler's approach does not seem to apply.

%
%

In what follows, we will say that an analytic function $\al(t)$ is
\textit{admissible} if it has the form $\al(t)=(1-t)\wt{\al}(t)$, where
$\sum_{n=0}^{\infty} \abs{\wt{\al}_n}<\infty$ and
$\wt{\al}$ is positive on $[0,1]$ (in particular, $\al_0>0$).

One of our main results is as follows.


\begin{thm}\label{T1 B 27 GL}
Let $T \in L(H)$ be an operator whose spectrum is contained in the
closed unit disc $\overline{\D}$. Let $\al(t)=(1-t)\wt{\al}(t)$ be
an admissible function such that $\sum_{n=0}^{\infty}
\abs{\wt{\al}_n} (1+\norm{T^n}^2) < \infty$. If $\al[T^*,T] \geq
0$, then $T$ is similar to a contraction.
\end{thm}

Notice that $\sum_{n=0}^{\infty} \abs{\wt{\al}_n} (1+\norm{T^n}^2)
< \infty$ implies that $\sum_{n=0}^{\infty} \abs{\al_n}
(1+\norm{T^n}^2) < \infty$, so that the operator $\al[T^*,T]$ is
well-defined.

If $\al(t)$ is an admissible function and an operator $T$ on $H$
is related to $\al$ as in the above theorem, then we will say that
$T$ belongs to the class $\cC_\al$ (see the definition at the beginning
of Section~\ref{sec-C-alpha}).

An important particular case is when $\al$ is analytic on a disc $|t|<R$ of radius $R>1$, in particular,
if $\al$ is a polynomial or is rational. In this case, $\al$ is
admissible whenever $\al(t)>0$ on $[0,1)$
and $\al(t)$ has a simple root at $t=1$. We get that given a function $\al$ of this type and
a Hilbert space operator $T$, whose spectral radius is less than or equal to $1$, $T$ is similar
to a contraction whenever $\al[T^*,T] \geq 0$.

We remark that the condition that $\tilde{\alpha}$ has no roots in
$[0,1]$ has a clear spectral meaning. Indeed, the eigenvalues and,
more generally, the approximate point spectrum of $T$ are
contained in $\{z\in \ol{\D}: \alpha(|z|^2)\ge 0\}$. As it is seen
from the example of normal operators, under the above condition,
the approximate point spectrum of $T$ can be whatever closed
subset of the closed unit disc.
\comm{
However, we do not know whether Theorem~\ref{T1 B 27 GL}
holds true if $\alpha$ has zeros on $(0,1)$.
}
\marginp{Preguntas - en otro sitio?}

The main difference with the approach originated by Agler is that here we do not need to assume
that $1/\al(\bar w z)$ is a reproducing kernel. That is, we admit that some of
the Taylor coefficients of $1/\al(t)$ at $t=0$ can be negative and we also allow
$\al$ to have zeros in $\overline\D$ (excluding the interval $[0,1]$). Notice that
we only get similarity to a contraction; in fact, there are many operators with
$\|T\|>1$, to which our results apply.
Notice that the majority of papers
based on Agler's approach (for the case of a single operator) deal only with
contractions (see the Remark~\ref{rem-contr} below).

Our main tool for proving Theorem~\ref{T1 B 27 GL} is related with Banach algebras.
We say that a sequence $\om = \{ \om_n \}_{n=0}^{\infty}$ of positive real numbers is a \emph{\gs} if
\begin{enumerate}
    \item[(GW 1)] $\om_n \geq 1$ for every $n$,
    \item[(GW 2)] $\om_n \om_m \geq \om_{n+m}$ (submultiplicative property) for every $n,m$,
    \item[(GW 3)] $\om_{n}^{1/n} \to 1$.
\end{enumerate}


Given a {\gs} $ \om$, we define the corresponding \emph{weighted Wiener algebra} $A_\om$ as the following set
of analytic functions:
\[
A_\om := \Big\{ \al(t) = \sum_{n=0}^{\infty} \al_n t^n  \, : \,
\sum_{n=0}^{\infty} \abs{\al_n} \om_n < \infty \Big\}.
\]
It is immediate to check that $A_\om$ is a commutative, unital Banach algebra
of analytic functions in $\overline\D$.


For analytic functions $f(t)=\sum_{n=0}^\infty f_n t^n$ and $g(t)=\sum_{n=0}^\infty g_n t^n$, we will
use the notation $f \succcurlyeq g$ when $f_n \geq g_n$ for every $n \geq 0$ and
the notation  $f \succ g$ when $f \succcurlyeq g$ and $f_0 > g_0$.
To prove Theorem~\ref{T1 B 27 GL}, the following lemma
on factorization in the algebra $A_\om$ will be used.

\begin{lemma}\label{L2 B 27 GL}
Let $\om$ be a {\gs}. If $f \in A_\om$ is a positive
function on $[0,1]$, then there exists a function $g \in A_\om$ such that $g \succ 0$ and $fg \succ 0$.
\end{lemma}

In Section~\ref{sec-C-alpha}, we
collect some elementary properties of classes $C_\al$;
Proposition~\ref{properties of class C_a} and Lemma~\ref{matrix without Jordan blocks}
give some examples. In particular, we show that any diagonalizable matrix
with spectrum on the unit circle belongs to $\cC_\al$ for some
admissible function~$\al$.

Given an operator $T$ of class $\mathcal{C}_\al$, in
\marginp{Por secciones?}
Section~\ref{sec-NF-model} we will write down its concrete coanalytic model, in other words, an
explicit Nagy-Foias-like functional model of $T$ up to similarity.

To construct a functional model
of a contraction, first one has to single out its unitary part
(recall that the Nagy-Foias construction ``forgets'' this part).
Section~\ref{sec-unitary} is devoted to defining
the unitary part of an operator $T\in \mathcal{C}_\alpha$,
which is a necessary first step to passing to the Nagy-Foias transcription.

The standard Nagy-Foias model of a contraction $S\in L(H)$
makes use of its defect operator, which is defined
as a nonnegative square root $D_S=(I-S^*S)^{1/2}$.
This model is related with the following well-known identity
\begin{equation}
\label{NF-identity}
\|h\|^2
=\sum_{n=0}^\infty \|D_S S^n h\|^2 + \lim_{n\to \infty}\|S^n h\|^2, \qquad h\in H,
\end{equation}
valid for any contraction $S$ (see \cite[Section 1.10]{NFBK10}).
This motivates the next definition, which will be useful for us.

\begin{defi}
Let $T\in L(H)$ be a power bounded operator
(that is, $\sup_{n\ge0}\|T^n\|<\infty$), and let $D:H\to F$,
where $H,F$ are Hilbert spaces. We will say that $D$ is
\emph{an abstract
defect operator for} $T$ if there are some positive constants
$c,C$ such that for any $h\in H$,
\begin{equation}
\label{def-abstr-def-oper}
c\|h\|^2
\le
\sum_{n=0}^\infty \|DT^n h\|^2 +\limsup_{n\to \infty}\|T^n h\|^2
\le
C\|h\|^2\, .
\end{equation}
\end{defi}

By applying the Banach limit, we will show that a
power bounded operator is similar to a contraction if and only if
it has an abstract defect operator. More precisely, we will prove the following.

%
%

\begin{lemma}
\label{lem-defect}
Let $T\in L(H)$ be a power bounded operator.
Then an operator $D \in L(H)$ is an abstract defect operator for $T$ if and only if there
exists an invertible operator $W\in L(H)$ such that  $\wt{T}:=WTW^{-1}$
is a contraction and
$\|Dh\|=\|D_{\wt{T}}Wh\|$ for any $h\in H$.
\end{lemma}

The Nagy-Foias-like model we give
is in some aspects close to \cite{Yak2004}.
However, here we deal with a general case and not only
with a $C_{0\,\bigcdot}$ case, as in \cite{Yak2004}.
Since the model is only up to similarity,
it is not unique, but we suggest a reasonable choice. It will be proven that
for an operator $T\in\cC_\al$,
$(\alpha[T^*, T])^{1/2}$ can be taken as an abstract defect operator of $T$.
This will permit us to write down explicitly an analogue of the characteristic function of $T^*$.

Our Nagy-Foias-like transcription implies
the major part of usual consequences of the Nagy-Foias theory
(such as criteria for completeness of eigenvectors of $T$ in terms
of the determinant of $\Theta_*$, criteria for these to form a Riesz basis, similarity
to a normal operator, etc).
These criteria are formulated in terms of the determinant of $\Theta_*(z)$.
In Section~\ref{sec-Sp-class} we show that, roughly speaking, $\Theta_*$ has a determinant
whenever $\alpha[T^*, T]$ is of
trace class and discuss briefly the above-mentioned consequences.

In Section~\ref{sec-cont}, we will give necessary and sufficient
conditions for the inclusion of operator classes $\cC_\al\subset
\cC_\tau$. It will follow, in particular, that there are many functions
$\tau$ such that the class
$\cC_\tau$ strictly contains $C_{1-t}$, the class of all contractions on $H$.

As compared with Agler's case, the construction of the Nagy-Foias model
in our case has some extra
difficulties. They are related with the fact that for operators of class $C_\al$
in general, the limit $\lim_n \|T^n h\|^2$, $h\in H$, does not exist
(and therefore we need in general Banach limits). In Section~\ref{sec-existence-limit}, we prove,
that these limits do exist
under the additional assumption that $\al$ has no roots on the unit circle (except
for the root at $t=1$).


\section{Proof of Lemma \ref{L2 B 27 GL} on factorization in the Banach algebra $A_\omega$}

If $\om_n=1$ for all $n$, then the
algebra $A_\om$ is just the usual Wiener algebra (which we denote $A_W$) of analytic functions
in $\D$ with absolutely summable Taylor coefficients. In fact, if we denote
by $\mathcal{H}(\overline{\D})$ the set of functions analytic on (a neighborhood of)
$\overline{\D}$, then, obviously,
\[
\mathcal{H}(\overline{\D}) \subset A_\om \subset A_W
\]
for every {\gs} $\om$. The first inclusion is due to the exponential decay of
Taylor coefficients
of functions in $\mathcal{H}(\overline{\D})$.

%

\begin{lemma}\label{L1 B 27 GL}
If $q(t) = (t-\lambda)(t-\overline{\lambda})$ for some $\lambda
\in \C \setminus \R$, then there exists a polynomial $p$ such that
$p \succ 0$ and $pq \succ 0$.
\end{lemma}

\begin{proof}
Let $m$ be the smallest nonnegative integer such that $\Re
(\lambda^{2^m}) \leq 0$. We define
\[
p(t):= \displaystyle \prod_{j=0}^{m-1}
(t^{2^j}+\lambda^{2^j})(t^{2^j}+\bar{\lambda}^{2^j})
\]
(so that $p(t) = 1$ if $m=0$). Note that by the minimality of $m$,
for each factor we have
\[
(t^{2^j}+\lambda^{2^j})(t^{2^j}+\bar{\lambda}^{2^j}) = t^{2^{j+1}}
+ 2 \Re (\lambda^{2^j}) t^{2^j} + \abs{\lambda^{2^j}} \succ 0.
\]
Therefore $p \succ 0$. Moreover
\[
(pq)(t) = (t^{2^m}-\lambda^{2^m})(t^{2^m}-\bar{\lambda}^{2^m}) =
t^{2^{m+1}} - 2\Re(\lambda^{2^m}) t^{2^m} + \abs{\lambda^{2^m}}
\succ 0. \qedhere
\]
\end{proof}

\begin{cor}\label{C1 B 27 GL}
If $q$ is a real polynomial without real roots and $q(0)>0$, then
there exists a polynomial $p$ such that $p\succ 0$ and $pq \succ
0$.
\end{cor}

\begin{proof}
Note that $q = Cq_1 \cdots q_k$, where $C$ is a positive constant
and each factor $q_j$ has the form $q_j (t) = (t - \lambda_j) (t -
\overline{\lambda}_j)$ for some $\lambda_j \in \C \setminus \R$.
Then for each factor $q_j$ we can construct the polynomial $p_j$
as in the previous lemma and we just take $p = p_1 \cdots p_k$.
\end{proof}

\begin{cor}\label{C2 B 27 GL}
If $q$ is a real polynomial such that $q(t)>0$ for every $t \in
[0,1]$, then there exists a function $u \in
\mathcal{H}(\overline{\D})$ such that $u\succ 0$ and $uq \succ 0$.
\end{cor}

\begin{proof}
Decompose $q$ as the product of polynomials $q_{nr}, q_+$ and $q_-$, where
the roots of $q_{nr}$ are nonreal, the roots of $q_+$ are positive
and the roots of $q_-$ are negative.
%
%
Without loss of generality, $q_+(0) = 1$, $q_-(0)=1$ and $q_{nr}(0) = q(0)>0$.
Therefore $q_- \succ 0$ and by Corollary \ref{C1 B 27 GL}, there exists a
polynomial $p$ such that $p \succ 0$ and $pq_{nr} \succ 0$.
Notice that $1/q_+\in \mathcal{H}(\overline{\D})$ and $1/q_+\succ 0$. Hence, we
can take $u := p/q_+$, and the statement follows.
\end{proof}


%

\begin{proof}[Proof of Lemma \ref{L2 B 27 GL}]
Let $f(t) > \varepsilon > 0$, for $t \in [0,1]$. Take $N \in \N$
such that $\sum_{n=N+1}^{\infty} \abs{f_n} \om_n < \varepsilon/2 $. Hence
it is obvious that $\sum_{n=N+1}^{\infty} \abs{f_n} < \varepsilon/2 $. Put
\[
f_N(t) = \sum_{n=0}^{N} f_n t^n - \frac{\varepsilon}{2}, \quad
h(t) =  \frac{\varepsilon}{2} + \sum_{n\ge N+1;\; f_n < 0}  f_n
t^n.
\]
Then $f_N$ is a polynomial and $h \in A_\om$. Since $\abs{f_N(t)} > \varepsilon
- \varepsilon/2 - \varepsilon/2 = 0$ for $t \in [0,1]$, we can apply Corollary \ref{C2 B 27 GL}
to obtain a function $u \in \mathcal{H}(\overline{\D})$ such that $u \succ 0$
and $uf_N \succ 0$. Hence $u \in A_\om$.

On the other hand, for $t \in \overline{\D}$ we have
\[
\Big|\sum_{n\ge N+1; \; f_n < 0}  f_n t^n\Big| \leq
\sum_{n=N+1}^{\infty} \abs{f_n} < \varepsilon/2.
\]
Hence $h(t) \not = 0$ for $t \in \overline{\D}$.
Notice that the properties of the weight imply that the characters of $A_\om$
are exactly the evaluation functionals at the
points of $\overline\D$.
\marginp{más explicación??}
It follows that $v:= 1/h \in A_\om$. Note that $v
\succ 0$, because it has the form $c/ (1-a)$, where $a \in A_\om$, $a
\succ 0$ and $c=\varepsilon/2$. Put $g:= uv \in A_\om$. Then $g \succ 0$ and since $f \succcurlyeq f_N + h$, we have
\[
gf \succcurlyeq g(f_N + h) = vuf_N + u \succ 0.  \qedhere
\]
\end{proof}


\section{The classes $\cC_\al$ and Proof of Theorem \ref{T1 B 27 GL}}
\label{sec-C-alpha}

In what follows, $\al(t) = (1-t)\wt{\al}(t)$
will be an admissible function. We associate to it the class of operators
\[
\mathcal{C}_\al := \{ T \in L(H): \quad \sigma(T)\subset \overline{\D}, \;
\sum_{n=0}^{\infty} \abs{\wt{\al}_n} (1+\norm{T^n}^2) < \infty,\; \al[T^*,T] \geq 0 \}.
\]
For example, $\cC_{1-t}$ is just the set of all contractions in $L(H)$.
Notice that any admissible function $\al\in\mathcal{H}(\overline{\D})$ has a simple root
at $t=1$, and the corresponding $\wt\al$ is also in $\mathcal{H}(\overline{\D})$. In this case,
any $T \in L(H)$ with
\marginp{Estoy quitando $r(T)$}
$\sigma(T)\subset \overline{\D}$ that satisfies $\al[T^*,T] \geq 0$
is in $\mathcal{C}_\al$.

Theorem \ref{T1 B 27 GL} asserts that
\emph{if $T \in \cC_{\al}$ for some admissible function $\al$ then $T$ is similar to a contraction}.

Here are some elementary properties of the classes $\cC_\al$.

\begin{prop}\label{properties of class C_a}
\begin{enumerate}
    \item[(a)] If $N \in L(H)$ is a normal operator with $\|N\|\le 1$, then $N \in \cC_\al$ for
    every admissible function $\al$. In particular, all unitary operators are in $\cC_\al$.
    \item[(b)] If $T_1, T_2 \in \cC_\al$, then
    the orthogonal sum $T_1 \oplus T_2$ also is in $\cC_\al$.
    \item[(c)] It $T \in \cC_\al$, then $\zeta T \in \cC_\al$ for every $\zeta$ on the unit circle $\T=\{z:|z|=1\}$.
    \item[(d)] If $T \in \cC_\al$, then $T | L \in \cC_\al$ for every $T$-invariant subspace $L \ss H$.
    \item[(e)] If $T$ is Hilbert space operator, whose spectral radius is less than one,
     then $T \in \mathcal{C}_{1-t^n}$ for any sufficiently large $n > 0$.
\end{enumerate}
\end{prop}

\begin{proof}
Statements (a)-(c) are obvious. (d) Put $S := T | L$. The first two conditions for $S \in \cC_\al$ follows from $\norm{S^n} \leq \norm{T^n}$, and the positivity condition is obvious since $\sum_{n=0}^{\infty} \al_n \norm{T^n h}^2 \geq 0$ for every $h \in H$, in particular for every $h \in L$. (e) Note that $\sigma(T) \subset \D$ implies that $\norm{T^n} < 1$ for $n \gg 0$.
\end{proof}

\begin{lemma}\label{matrix without Jordan blocks}
For any complex square matrix $T$ without nontrivial Jordan
blocks such that $\sigma(T) \ss \T$, there exists an
admissible polynomial $p(t)$ such that $p[T^*,T] = 0$.
\end{lemma}

\begin{proof}
Suppose $T$ is of size $n\times n$.
Let $\{ v_j \}$ $(1\le j\le n)$ be a basis of
eigenvectors of $T$ in $\C^n$ and let $\la_j\in\T$ be the corresponding eigenvalues.
Consider the polynomial
\[
p(t) = (1-t) \prod_{1 \leq k < \ell \leq n} (1- 2 \Re(\la_k\bar \la_\ell)t + t^2).
%
%
\]
Notice that it only has roots on the unit circle; it follows that $p$ is admissible.
We will prove that $p[T^*,T] = 0$. Each $h \in \C^n$ can be written as $h =
\sum h_j v_j$. We get
\[
\begin{split}
\big\langle p[T^*,T]h,h \big\rangle &= \sum_j p_j \Big\langle T^j \sum_k h_k v_k, T^j \sum_\ell h_\ell v_\ell \Big\rangle
= \sum_j p_j \Big\langle \sum_k \lambda_{k}^{j} h_k v_k, \sum_\ell \lambda_{\ell}^{j} h_\ell v_\ell \Big\rangle \\
&= \sum_{k,\ell} \sum_j p_j  \overline{\lambda_{\ell}^{j}} \lambda_{k}^{j}\, h_k \bar h_\ell \langle v_k,v_\ell \rangle
= \sum_{k,\ell} p(\lambda_k \bar\lambda_\ell) h_k
\bar h_\ell \langle v_k,v_\ell \rangle = 0,
\end{split}
\]
because $p(\lambda_k \overline{\lambda_\ell})=0$ for all $k,\ell$.
\end{proof}


For the proof of Theorem \ref{T1 B 27 GL} we need some technical results. Let $\om$ be a {\gs} and let $f \in A_\om$. If $T,B \in L(H)$ and $T$ satisfies the condition $\norm{T^n}^2 \lesssim \om_n$ (i.e., $\norm{T^n}^2 \leq C \om_n$ for some positive constant $C$), then the operator
\[
f[T^*,T](B) := \sum_{n=0}^{\infty} f_n T^{*n}BT^n
\]
is well defined. Indeed, $\norm{f[T^*,T](B)} \lesssim \norm{B} \norm{f}_{A_\om}$. Note that, in particular, $f[T^*,T](I) = f[T^*,T]$.

\begin{lemma}\label{L3 B 27 GL}
Let $\om$ be a {\gs}, $f,g,h \in A_\om$ such that $fg=h$ and let $T,B \in
L(H)$. If $\norm{T^n}^2 \lesssim \om_n$ then one has
\begin{enumerate}
\item[\textnormal{\textbf{(i)}}] \enspace
$
h[T^*,T](B) = g[T^*,T](f[T^*,T](B))
$.

\item[\textnormal{\textbf{(ii)}}] \enspace
$
h[T^*,T] = g[T^*,T](f[T^*,T])
$.
\end{enumerate}
\end{lemma}

\begin{proof}
Let us define
$$
g_{[N]}(t) := \sum_{n=0}^{N} g_n t^n \quad \textrm{and} \quad
h_{N}:= f g_{[N]}   \quad (N \geq 0).
$$
Then, $\| g - g_{[N]} \|_{A_\om} \xrightarrow[N \to
\infty]{} 0$ and $ \| h - fg_{[N]} \|_{A_\om} = \| fg
- fg_{[N]} \|_{A_\om} \xrightarrow[N \to \infty]{} 0
$.
It easily implies that
\begin{equation}
\label{1a}
\| (g_{[N]}f)[T^*,T](B) - h[T^*,T](B) \|_{L(H)} \xrightarrow[N \to
\infty]{} 0.
\end{equation}
Note that $(z^n f)[T^*,T](B) = T^{*n} f[T^*,T](B) T^n$ for every
$n \geq 0$. Hence
$$
(g_{[N]}f)[T^*,T](B) = \sum_{n=0}^{N} g_n T^{*n} f[T^*,T](B) T^n,
$$
and therefore,
\begin{equation}
\label{1b}
\| (g_{[N]}f)[T^*,T](B) - g[T^*,T](f[T^*,T](B))
\|_{L(H)} \xrightarrow[N \to \infty]{} 0.
\end{equation}
Formulas \eqref{1a} and \eqref{1b} give (i). To get (ii), one just has to put $B=I$.
\end{proof}

If $A,T \in L(H)$ and $A$ is positive, it is said that $T$
\emph{is an $A$-contraction} if $T^* A T \leq A$.

\begin{rem}\label{R1 B 27 GL}
$T$ is similar to a contraction if and only if $T$ is an
$A$-contraction for some $A \geq \varepsilon I$,
where $\varepsilon >0$.
\end{rem}

\begin{proof}[Proof of Theorem \ref{T1 B 27 GL}]
Put $\om_n=1+\|T^n\|^2$. Since $\|T^{m+n}\|\le \|T^{m}\|\|T^{n}\|$,
the weight $\om$ satisfies (GW 1) -- (GW 3). By Lemma \ref{L2 B 27 GL},
there exists a function $\wt{\be} \in A_W$ such that $\wt{\be} \succ 0$ and
$f := \wt{\be} \wt{\al} \succ 0$. Then $(1-t) f = \wt{\be}(t) \al(t)$ and by Lemma \ref{L3 B 27 GL} (ii) we get
\begin{equation}
\label{des-operadores}
f[T^*,T] - T^* f[T^*,T] T =
\sum_{n=0}^{\infty} \wt{\be}_n T^{*n} \al[T^*,T] T^n \geq 0.
\end{equation}
Hence $T$ is a $f[T^*,T]$-contraction and the theorem
follows from Remark \ref{R1 B 27 GL}.
\end{proof}

\comm{
In a similar way, one gets the following (see~\cite[Theorem 3.4 and Corollary 4.5]{Olo15})

\begin{thm}\label{T1new}
Let $T \in L(H)$ be an operator whose spectrum is contained in the
\marginp{\bf !!! QUITAR}
closed unit disc $\overline{\D}$. Let $\al(t)=(1-t)^a\wt{\al}(t)$ be
a function such that $a>0$ and $\sum_{n=0}^{\infty}
\abs{\wt{\al}_n} (1+\norm{T^n}^2) < \infty$. If $\al[T^*,T] \geq
0$, then $T$ is similar to an operator $\tilde T$ such that
$(1-t)^a[\tilde T^*, \tilde T]\ge 0$.
\end{thm}
\marginp{$T$ similar to a contraction?}
}


%
%
%


\section{The abstract defect operator of $T$}
\label{sec-abstr-def-oper}

Let us begin by recalling the notion of a Banach limit. If we denote by $c$ the set of all convergent complex sequences then we can define the linear functional $L : c \to \C$ given by $L(x) = \lim x_n$ for every $x = \{ x_n \}_{n=1}^{\infty} \in c$. It is immediate that $\norm{L} = 1$, $L(x') = L(x)$ if $x' = \{ x_n \}_{n=2}^{\infty}$, and also $L(x) \geq 0$ if $x \geq 0$ (i.e., $x_n \geq 0$ for every $x$). Using the Hahn-Banach Theorem, these properties of the limit functional can be extended to $\ell^{\infty}$.

\begin{thmA}\label{Thm A Banach limit}
There is a linear functional $L : \ell^{\infty} \to \C$
such that
\begin{enumerate}
    \item[\tn{(a)}] $\norm{L} = 1$;
    \item[\tn{(b)}] $L(x) = \lim x_n$ for every $x \in c$;
    \item[\tn{(c)}] $L(x) \geq 0$ for every $x \in \ell^\infty$ such that $x \geq 0$;
    \item[\tn{(d)}] $L(x') = L(x)$ if $x \in \ell^\infty$ and $x' = \{ x_n \}_{n=2}^{\infty}$;
    \item[\tn{(e)}] $\liminf x_n \leq L(x) \leq \limsup x_n$ if $x \in \ell^\infty$ is a real sequence.
\end{enumerate}
\end{thmA}

\begin{proof}
Statements (a)-(d) are contained in \cite[Theorem III.7.1]{Con90} and assertion (e) is their easy consequence (and it is also standard).
\end{proof}

A functional $L$ with the above properties is called a Banach limit.

\begin{proof}[Proof of Lemma \ref{lem-defect}]
We remark first that by a lemma by Gamal \cite[Lemma 2.1]{Gam11}, for any power bounded operator $T$,
one has
\begin{equation}
\label{gamal}
\liminf_{n\to\infty} \|T^n h\|^2 \asymp
\limsup_{n\to\infty} \|T^n h\|^2, \quad h\in H
\end{equation}
(we say that two quantities $A,B$,
depending on $h$ or some other parameter, are comparable and write $A \asymp B$
if there are two positive constants $c,C$ such that $c A \leq B \leq C A$).

Suppose first that there exists a linear isomorphism $W \in L(H)$
with the properties stated in Lemma.
Since $\wt{T} = WTW^{-1} \in L(H)$ is a contraction,
by \eqref{NF-identity}
we have
\[
\norm{h}^2 = \sum_{n=0}^{\infty} \norm{D_{\wt{T}} \wt{T}^n h}^2 + \lim_{n\to \infty} \norm{\wt{T}^n h}^2.
\]
Note that $\wt{T}^n = WT^n W^{-1}$, and thus
\[
\norm{h}^2 = \sum_{n=0}^{\infty} \norm{D_{\wt{T}} WT^nW^{-1} h}^2 + \lim_{n\to \infty} \norm{WT^nW^{-1} h}^2.
\]
Since $W$ is invertible and $\norm{Dh} = \norm{D_{\wt{T}} Wh}$, we get
\begin{equation}\label{eq. equivalent norms}
\norm{h}^2 \asymp \norm{Wh}^2 = \sum_{n=0}^{\infty} \norm{DT^n h}^2 + \lim_{n\to \infty} \norm{WT^n h}^2.
\end{equation}

By \eqref{gamal}, $\lim_{n\to \infty} \norm{WT^n h}^2 \asymp \limsup_{n\to \infty} \norm{T^n h}^2$. We
deduce that $D$ is an abstract defect operator for $T$.

Conversely, suppose now that $D$ is an abstract defect operator for $T$. Fix a Banach limit $L$
and put

\begin{equation}\label{eq. new norm sum banach limit}
\Norm{h}^2 := \sum_{n=0}^{\infty} \norm{DT^nh}^2 + L\big(\{\norm{T^nh}^2\}\big).
\end{equation}

Notice that \eqref{gamal} and Theorem~A (e) give that
\[
\limsup_{n\to\infty} \|T^n h\|^2 \asymp
L\, (\{\|T^nh\|^2\}),
 \quad h\in H.
\]
The relation \eqref{def-abstr-def-oper} implies that $\Norm{h}\asymp \|h\|$, $h\in H$.
It follows that $\Norm{\cdot}$ is an equivalent Banach space norm on $H$.
By applying the Cauchy-Schwarz inequality, it is easy to see that
\[
[x,y] := \sum_{n=0}^{\infty} \langle DT^nx, DT^ny\rangle + L \big(\{\langle T^nx,T^ny \rangle\}\big)
\]
absolutely converges for any $x,y\in H$. It is a semi-inner product on $H$,
which induces the norm $\Norm{\cdot}$. So, in fact, $\Norm{\cdot}$ is
a Hilbert space norm equivalent to $\norm{\cdot}$.
(see \cite{Ker89} and \cite{Nagy47} for a similar argument).

Therefore there exists a linear isomorphism $W : H \to H$ such that $\norm{Wh} = \Norm{h}$. Observe that
\[
\Norm{Th}^2 = \Norm{h}^2 - \norm{Dh}^2 \leq \Norm{h}^2.
\]
Let $\wt{T} := WTW^{-1} \in L(H)$ (similar to $T$). Take $x \in H$ and put $h := W^{-1}x$. We get
\[
\norm{WTh} \leq \norm{Wh}
\]
so $\wt{T}$ is a contraction. Since $\norm{Dh}^2 = \Norm{h}^2 - \Norm{Th}^2$ and
\[
\norm{D_{\wt{T}}Wh}^2 = \norm{Wh}^2 - \norm{\wt{T}Wh}^2 = \Norm{h}^2 - \Norm{Th}^2,
\]
we get $\norm{Dh} = \norm{D_{\wt{T}}Wh}$ for every $h \in H$.
\end{proof}

%

\

Let $\al$ be an admissible function and let $T \in \mathcal{C}_\al$.
We know already that $T$ is similar to a contraction. Since
$\al \in A_W$,
by Lemma \ref{L2 B 27 GL},
there exists a function $\wt{\be} \in A_W$ such that $\wt{\be} \succ 0$
and $f := \wt{\be}\wt{\al} \succ 0$. Hence $(1-t)f(t) = \wt{\be}(t) \al(t)$.
Set
\[
B:= (f[T^*,T])^{1/2},
\]
where the positive square root has been taken. Then $B > \ep I$ for some $\ep>0$.
We will assume, without loss of generality, that $\sum f_k = \|f\|_{A_W}=1$.
We put
\begin{equation}
\label{eqn-defect-op-D}
D := (\alpha[T^*,T])^{1/2}.
\end{equation}
\begin{thm}\label{thm new norm}
If $T \in \mathcal{C}_\al$ for some admissible function $\al \in A_W$, then $D$
is an abstract defect operator for $T$. More specifically, if
$\wt{\be},f$ and $B$ are as above, then the expression
\begin{equation}\label{eq. new norm}
\Norm{h}^2 := \sum_{n=0}^{\infty} \norm{DT^n h}^2  + \lim_{n \to \infty} \norm{BT^nh}^2
\end{equation}
defines an equivalent Hilbert space norm in $H$ and $T$ is a contraction with
respect to this norm. In particular, the limit in \eqref{eq. new norm} exists for every $h \in H$. Moreover,
\begin{equation}
\label{eq. new norm eq defect}
\Norm{h}^2 - \Norm{Th}^2 = \norm{Dh}^2 \quad (\forall h \in H).
\end{equation}
\end{thm}

\begin{proof}
Since $(1-t)f(t) = \wt{\be}(t) \al(t)$, we have
\[
B^2 - T^* B^2T = \sum_{n=0}^{\infty} \tilde{\beta}_n T^{*n}D^2
T^n.
\]
Therefore, for every $h \in
H$ we have
\[
\norm{B h}^2 - \norm{B Th}^2 = \sum_{n=0}^{\infty} \tilde{\beta}_n
\norm{DT^n h}^2.
\]
Changing $h$ by $T^j h$ we obtain
\[
\norm{B T^j h}^2 - \norm{B T^{j+1}h}^2 = \sum_{n=0}^{\infty}
\tilde{\beta}_n \norm{DT^{n+j} h}^2,
\]
for every $j \ge 0$. Summing these equations for $j = 0, 1,
\ldots , N-1$ we obtain
\begin{equation}\label{eq. 2n}
\norm{B h}^2 - \norm{B T^N h}^2 = \sum_{n=0}^{N-1} \beta_n
\norm{DT^n h}^2 + \sum_{n=N}^{\infty} \left( \sum_{j=n-N+1}^{n}
\tilde{\beta_j}  \right) \norm{D T^n h}^2,
\end{equation}
where $\be_n= \sum_{j=0}^{n} \wt{\be}_j$. In particular,
$0<\wt{\be}_0\le \be_n$ and therefore by \eqref{eq. 2n} we get
\[
\norm{B h}^2 \geq \sum_{n=0}^{N-1} \beta_n \norm{DT^n h}^2 \geq \wt{\be}_0 \sum_{n=0}^{N-1} \norm{DT^n h}^2.
\]
Hence the series $\sum_{n=0}^{\infty} \norm{DT^n h}^2$ converges. On the other hand, since
\[
\sum_{j=n-N+1}^{n}\tilde{\beta_j} \leq \sum_{j=0}^{\infty}\tilde{\beta_j} = \norm{\wt{\be}}_{A_W} < \infty,
\]
we obtain that
\[
\sum_{n=N}^{\infty} \left( \sum_{j=n-N+1}^{n}\tilde{\beta_j}  \right) \norm{D T^n h}^2 \leq \norm{\wt{\be}}_{A_W}
\sum_{n=N}^{\infty} \norm{D T^n h}^2 \to 0
\]
when $N \to \infty$. Therefore, taking limit in \eqref{eq. 2n} when $N$ goes to infinity  we obtain that
\begin{equation}
\label{NF-ident-Bh2}
\norm{B h}^2 =
\sum_{n=0}^{\infty} \be_n \norm{DT^n h}^2 + \lim_{N \to \infty} \norm{B T^N h}^2
\end{equation}
(and the limit in the right hand side exists for any $h$).
Since $\wt{\be}_0\le \be_n \le \|\wt\be \|_{A_W}<\infty$, it follows that
$\Norm{h}$ defines an equivalent Hilbert space norm on $H$. Formula \eqref{eq. new norm eq defect}
is immediate from \eqref{eq. new norm}.
\end{proof}

\begin{rem}
Notice that Theorems~\ref{T1 B 27 GL} and~\ref{thm new norm} give two methods of finding an equivalent
norm such that $T$ is a contraction in this norm. With the method of Theorem~\ref{thm new norm} we
obtained in addition that $D$ is an abstract defect operator.
\end{rem}

%

\begin{rem}
\label{rem-contr}
Assume that
$\al(t)\ne 0$ for $t\in \overline\D$, $t\ne 1$, and
$1/\wt\al$ has positive Taylor coefficients at the origin.
Then in the above calculations, we could set
$f(t)\equiv 1$, so that $B=I$ and
$\wt \be=1/\wt\al\succ 0$.
In this case,
formula~\eqref{NF-ident-Bh2}
yields a unitarily equivalent (coanalytic) model of $T$, which represent it as a part of an operator
$\bigoplus_{j=1}^\infty M^*\oplus U$, where $M$ is a weighted shift and $U$ is unitary.
In this situation, there is no need
to assume that $\al$ is admissible and one can deal with more general functions.
For the case when $\al$ is a polynomial, this a result by M\"uller \cite[Theorem 3.10]{Mul88}.
Most general result of this kind was given recently
by Olofsson, see Theorem 6.6 in \cite{Olo15}.
On the other hand, in this case $\wt \be\succ 0$, so that $\{\be_n\}$ is an increasing sequence and
therefore the backward shift $M^*$ has to be a contraction.

In fact, in this setting, one has not
to assume that $\al$ is admissible and can deal with more general functions. However,
in both results cited above, $T$ has to be a contraction.
Theorem~\ref{thm new norm} requires $\al$ to be admissible, but applies
to operators $T$ that are not contractions.
\marginp{!!!!!!!!!}
\end{rem}
\begin{rem}\label{remark on lim*}
Using the notation above, note that
\[
\norm{BT^nh}^2 = \sum_{k=0}^{\infty} f_k \left\langle T^{n+k}h, T^{n+k}h \right\rangle.
\]
So if we put
\begin{equation}
\label{lim-st}
\tn{lim}^* x_n =
\lim_{n \to \infty} \sum_{k=0}^{\infty} f_k x_{n+k},
\end{equation}
then \eqref{eq. new norm} can be written as
\begin{equation}
\label{eq-lim-st}
\Norm{h}^2 = \sum_{n=0}^{\infty} \norm{DT^n h}^2 +
\tn{lim}^*_{n \to \infty} \norm{T^nh}^2.
\end{equation}
So, on the contrary to the general formula \eqref{eq. new norm sum banach limit}, the use of a
Banach limit is unnecessary in the context of an operator $T$ in the class $\cC_\al$, and one can
use instead the ``regularized'' limit $\tn{lim}^*$. We observe that $\tn{lim}^*$ coincides with the usual limit when
the sequence is convergent (here we use the above normalization assumption that $\sum f_k = 1$).

On the other hand, the example
of matrices $T$, meeting the requirements of Lemma~\ref{matrix without Jordan blocks},
shows that in general, the limit $\lim \norm{T^n h}^2$ does not exist for $T \in \cC_\al$. In
Section~\ref{sec-existence-limit},
we will show that this limit does exist if $\al$ satisfies an extra requirement.
\end{rem}

In what follows, for $T\in \cC_\al$, the operator
$D$, given by~\eqref{eqn-defect-op-D}, will be called
\emph{the defect operator of} $T$.


\section{On definition of unitary part of a $\cC_\al$ operator}
\label{sec-unitary}


\begin{defi}
Let $K$ be a Hilbert space. We say that $T \in L(K)$ is a
\emph{completely nonunitary operator} if there is no nonzero
reducing subspace $L$ for $T$ such that $T | L$ is unitary.
\end{defi}

It is well-known that every contraction $S$ can be decomposed into an
orthogonal sum of a  unitary operator and a completely nonunitary
operator (called the \emph{unitary part} and the \emph{completely
nonunitary part} of $S$, respectively).
We recall that the standard construction of the Nagy-Foias model
applies only to completely nonunitary contractions.

If one takes an operator
$T$ in the class $\mathcal{C}_\al$ and applies to
it Theorem~\ref{thm new norm},
then one gets a \emph{direct sum} decomposition
\begin{equation}\label{eq. H = H0 + H1}
H=H_0\dotplus H_1
\end{equation}
such that
$T|H_0$ is similar to a unitary operator and
$T|H_1$ is similar to a completely non-unitary operator.

For a general admissible function $\alpha$, we cannot say much more.
%
%
However, some extra properties hold if $\al$ is in following subclass.

\begin{defi}
A function $\al \in A_W$ will be called
\emph{strongly admissible} if it has the form $\al(t) = (1-t) \wt{\al}(t)$ for
some function $\wt{\al} \in A_W$ with real Taylor coefficients, which
has no roots on the unit circle $\T$
and satisfies $\wt{\al}(0)=\al(0)>0$.
\end{defi}

Notice that any strongly admissible function is admissible.

For the sequel, let us recall the following characterization of the unitary
part of a contraction.

\begin{thmB}[See \cite{NFBK10}, Theorem I.3.2]
To every contraction $S$ on the space $H$ there corresponds a
decomposition of $H$ into an orthogonal sum of two subspaces
reducing $S$, say $H = H_0 \oplus H_1$, such that the part of $S$
on $H_0$ is unitary, and the part of $S$ on $H_1$ is completely
nonunitary; $H_0$ or $H_1$ may equal the trivial subspace $\{ 0
\}$. This decomposition is uniquely determined. Indeed, $H_0$
consists of those elements $h$ of $H$ for which
\[
\norm{S^n h} = \norm{h} = \norm{S^{*n}h} \quad \quad (n=1,2,
\ldots).
\]
$S_0 = S | H_0$ and $S_1 = S | H_1$ are called the unitary part
and the completely nonunitary part of $S$, respectively, and $S =
S_0 \oplus S_1$ is called the canonical decomposition of $S$.
\end{thmB}

The next theorem is an analogue of this decomposition
for our operators in the case of a strongly admissible
function $\alpha$.

\begin{thm}\label{Theorem 4}
Let $T \in \mathcal{C_\alpha}$, where $\alpha$ is strongly admissible.
Denote by $H_0$ the elements $h
\in H$ for which there exists a two-sided sequence $\{ h_n \}_{n
\in \Z}$ such that $h_0 = h, Th_n = h_{n+1}$ and $\norm{h_n} =
\norm{h}$ for every $n \in \Z$. Let $\wT$ be the operator $T$
acting on $\wH := (H, \Norm{\cdot})$ (the new norm, which was given in
\eqref{eq. new norm}). Then $H_0$ is a closed subspace of $H$ and there
exists a direct sum decomposition $H = H_0 \dotplus H_1$ with the
following properties:
\begin{enumerate}[label= \textnormal{(\roman*)}]
    \item $T | H_0$ is unitary;
    \item $H_0$ and $H_1$ are invariant subspaces of $T$;
    \item $H_0$ and $H_1$ are orthogonal in $\wH$;
    \item $\wT | H_0$ and $\wT | H_1$ are the unitary part and the completely nonunitary part
    of the contraction $\wT$, respectively.
\end{enumerate}
\end{thm}

\begin{rem}
In Lemma \ref{matrix without Jordan blocks} we saw that if $T$ is a finite matrix without
Jordan blocks and $\sigma(T) \ss \T$, then $T \in C_p$ for some admissible polynomial $p$.
In this case, $H_0 = H$ and $H_1 = H$ in the decomposition \eqref{eq. H = H0 + H1}, but $T | H_0 = T$ is non-unitary,
as a rule. As a consequence we get:
\begin{enumerate}
    \item[(i)] Theorem \ref{Theorem 4} is not valid if we do require $\al$ to be strongly admissible;
    \item[(ii)] A non-unitary finite matrix $T$ with $\sigma(T) \ss \T$ cannot belong
to $\cC_\al$ if $\al$ is strongly admissible.
  \end{enumerate}
\end{rem}

Observe that a completely nonunitary contraction can be similar to a unitary operator. For
a general operator $T$, one cannot single out the largest
direct summand which is similar to a unitary.

\begin{proof}[Proof of Theorem \ref{Theorem 4}]
Since $\wT$ is a contraction on $\wH$, Theorem B gives us that for
the decomposition $\wH = \wH_0 \oplus \wH_1$, where $\wH_0$
consists of those elements $h$ of $\wH$ for which
\[
\Norm{\wT^n h} = \Norm{h} = \Norm{\wT^{*n}h} \quad \quad (n=1,2,
\ldots),
\]
we have that $\wT_0 := \wT | \wH_0$ is unitary and $\wT_1 := \wT |
\wH_1$ is completely nonunitary.

The goal of the following two claims is to prove that $H_0 =
\wH_0$.

\textbf{Claim 1.} \textit{Let $h \in H$. Then $h \in \wH_0$ if and
only in there exists a sequence $\{ h_n \}_{n \in \Z}$ such that
$h_0 = h, \wT h_n = h_{n+1}$ and $\Norm{h_n} = \Norm{h}$ for every
$n \in \Z$.}

(In fact, this claim is a variation of Theorem B and is valid for
any contraction $\wT$ on a Hilbert space.)

Indeed, suppose that $h \in \wH_0$. Define the sequence $\{ h_n
\}_{n \in \Z}$ by $h_0 := h, h_n := \wT^n h$ and $h_{-n} :=
\wT^{*n}h$, for $n \geq 1$. Since $\wT | \wH_0$ is unitary we
obtain that $\wT h_{-n} = \wT \wT^{*n}h = \wT^{*(n-1)} = h_{-n+1}$
for $n \geq 1$. It follows that the sequence $\{ h_n \}_{n \in
\Z}$ satisfies the conditions of the statement.

Reciprocally, suppose that a sequence $\{ h_n \}_{n \in \Z}$
satisfies the conditions of the statement. Fix $n \geq 1$. Since
$\wT^{n} h_{-n} = h$, we have that $\Norm{\wT^n h_{-n}}^2 =
\Norm{h}^2 = \Norm{h_{-n}}^2$. Hence $\langle (I-
\wT^{*n}\wT^n)h_{-n} , h_{-n} \rangle = 0$. But using that $\wT^n$
is a contraction on $\wH$ it follows that $(I-
\wT^{*n}\wT^n)h_{-n} = 0$. Therefore, using that $\wT^{n} h_{-n} =
h$, we obtain that $\wT^{*n} h = h_{-n}$. Then $h \in \wH_0$.
This finishes the proof of Claim 1.

\textbf{Claim 2.} \textit{Let $h \in H$. Then $h \in \wH_0$ if and
only in there exists a sequence $\{ h_n \}_{n \in \Z}$ such that
$h_0 = h, Th_n = h_{n+1}$ and $\norm{h_n} = \norm{h}$ for every $n
\in \Z$.}

(Note that this claim is stated in terms of the original norm in
$H$.)

Indeed, we apply Claim 1. Let $\{ h_n \}_{n \in \Z}$ be a sequence
such that $h_0 = h$ and $Th_n = h_{n+1}$.
We have to show that the sequence $\{\Norm{h_{n}}^2\}_{n\in \Z}$
is constant if and only if
the sequence $\{\norm{h_{n}}^2\}_{n\in \Z}$ is constant.
Since the two norms are equivalent, if either of these two sequences is constant,
the other is in $\ell^\infty(\Z)$.

By \eqref{eq. new norm eq defect} we have that
\[
\Norm{h_{n+1}}^2 = \Norm{T h_n}^2 = \Norm{h_n}^2 - \norm{Dh_n}^2.
\]
Therefore $\Norm{h_n} = \Norm{h}$ for every $n \in \Z$ if and only
if $\norm{Dh_n}^2 = 0$ for every $n \in \Z$.
We will use the backward shift operator $\nabla$,
acting on $\ell^\infty(\Z)$
\marginp{He agregado esto}
by $[\nabla a]_n = a_{n+1}$, $a\in\ell^\infty(\Z)$.
For any $f \in A_W$, $f(\nabla)$ is well-defined by
$\big(f(\nabla) a\big)_n:= \sum_{j=0}^{\infty} f_j a_{n+j}$.

Denote by
\marginp{Revisar!}
$\hh$ the sequence $\{ \norm{h_n}^2 \}_{n \in \Z}$. Then it
is easy to obtain that $\norm{Dh_n}^2 =
[\alpha(\nabla)\hh]_n$. Hence $\Norm{h_n} = \Norm{h}$ for
every $n \in \Z$ if and only if $\alpha(\nabla)\hh =
(\ldots, 0,0,0, \ldots) = \textbf{0}$.

So we have to show that $\hh$ is constant if and only if
$\alpha(\nabla) \hh = \textbf{0}$. The direct
implication is obvious. For the converse, fix a factorization $\alpha(t) = (1-t)
q(t) \gamma (t)$, where $q$ is a polynomial with zeroes in $\D$
and $\gamma \in A_W(\overline{\D})$ without zeroes in
$\overline{\D}$ (recall that $\al$ is assumed to be strongly admissible). Then
$1/\gamma \in A_W(\overline{\D})$.
Hence $\gamma(\nabla) [ q(\nabla) (1- \nabla)\hh] =
\textbf{0}$ implies that $ q(\nabla) (1- \nabla)\hh = \textbf{0}$
(just multiply by $1/\gamma$). Now let $\bfg =
(1-\nabla)\hh$. We want to proof that $\bfg =
\textbf{0}$. When $q$ has just a single root, say $q(t) = 1-at$
for some $a$ with $\abs{a} > 1$, the result is immediate.
For a general $q$
we just need to apply induction on the number of roots of $q$.
This finishes the proof of Claim 2.

Therefore we have that $H_0 = \wH_0$. Put $H_1 := \wH_1$. Now (i)
is obvious, since $T | H_0$ is a surjective isometry. The rest of
items of the theorem follow immediately using Theorem~B.
\end{proof}

\begin{rem}
In Lemma \ref{matrix without Jordan blocks} we saw that if $T$ is a finite matrix without
Jordan blocks and $\sigma(T) \ss \T$, then $T \in C_p$ for some admissible polynomial $p$.
In this case, $H_0 = H$ and $H_1 = H$ in the decomposition \eqref{eq. H = H0 + H1}, but $T | H_0 = T$ is non-unitary,
as a rule. As a consequence we get:
\begin{enumerate}
    \item[(i)] Theorem \ref{Theorem 4} is not valid if we do require $\al$ to be strongly admissible;
    \item[(ii)] A non-unitary finite matrix $T$ with $\sigma(T) \ss \T$ cannot belong
to $\cC_\al$ if $\al$ is strongly admissible.
  \end{enumerate}
\end{rem}


\section{The Nagy-Foia\c{s} model of $T$}
\label{sec-NF-model}
Let $T \in \mathcal{C}_\alpha$ for some $\alpha \in
A_W$.
\marginp{\textbf{Dejamos, de momento}}
In this section, for simplicity, we
will assume that $\al$ is strongly admissible and
$T$ is a completely nonunitary operator.


Then, using the notation of the previous section, we have that
$\widetilde{T}$ is a completely nonunitary contraction on
$\widetilde{H}$. Let $D_{\widetilde{T}}$ and $D_{\widetilde{T}^*}$
be the defect operators of $T$ and let
$\mathcal{D}_{\widetilde{T}}$ and $\mathcal{D}_{\widetilde{T}^*}$
be its defect spaces.  We recall that the defect operator
of $T$ has been defined by~\eqref{eqn-defect-op-D}. We define the
\emph{defect space of} $T$ as
$\mathcal{D}_T = \operatorname{clos} DH$, where the closure is taken with respect to
$\norm{\cdot}$.

Define $V : \mathcal{D}_{\widetilde{T}} \to \mathcal{D}_T$ by
$V(D_{\widetilde{T}}h) = Dh$, $h \in H$. Then, using equation
\eqref{eq. new norm eq defect}, we obtain that $V$ is an isometry from
$\mathcal{D}_{\tilde{T}}$ to $\mathcal{D}_T$. Let us define the
functions $\Theta_{*} \in H^{\infty}(\mathcal{D}_{\widetilde{T}^*}
\to \mathcal{D}_T)$ and $\Delta_* : \T \to
\mathcal{D}_{\widetilde{T}^*}$ given by
\[
\Theta_{*}(z)h := V (-\widetilde{T}^* +
zD_{\widetilde{T}}(I-zT)^{-1}D_{\widetilde{T}^*}) h, \quad h \in
\mathcal{D}_{\widetilde{T}^*},\; z\in \D,
\]
\[
\Delta_*(\zeta) := (I - \Theta_*(\zeta)^*\Theta_*(\zeta))^{1/2},
\quad \zeta \in \T.
\]
Now put
\[
\mathcal{K}_{\Theta_{*}} := { H^2(\mathcal{D}_T) \choose
\operatorname{clos} \Delta_* L^2(\mathcal{D}_{\widetilde{T}^*}) } \ominus
{ \Theta_* \choose \Delta_* } H^2(\mathcal{D}_{\widetilde{T}^*}).
\]
Finally, let us define $\Phi_1 : H \to H^2(\mathcal{D}_T)$ by
$\Phi_1 h(z) = D(I-zT)^{-1}h$ and $M_* : \mathcal{K}_{\Theta_{*}}
\to \mathcal{K}_{\Theta_{*}}$ by
\[
M_* {u \choose v} := {\frac{u(z) - u(0)}{z} \choose z^{-1} v(z)}.
\]

\marginp{Discusión de tipos de función caracteristica. Diferencias con \cite{Yak2004}}

\begin{thm}
Let $T\in\cC_\al$, where $\al$ is strongly admissible.
With the notation used above, there exists a linear map $\Phi_2 : H
\to \operatorname{clos} \Delta_* L^2(\mathcal{D}_{\widetilde{T}^*})$ such
that $\norm{\Phi_2 h(z)}^2 = \tn{lim}^*_{n\to \infty} \norm{T^n h}^2$
for every $h \in H$, and
\begin{enumerate}[label= \textnormal{(\roman*)}]
    \item $\Phi:= {\Phi_1 \choose \Phi_2} : H \to \mathcal{K}_{\Theta_{*}}$ is an isometric isomorphism;
    \item $\Phi T = M_* \Phi$.
\end{enumerate}
\end{thm}

It is a common point that this kind of result implies a variant of von Neumann inequality: if $T$
satisfies the conditions of the above theorem, then for any $f\in \mathcal{H}(\overline\D)$,
\[
\norm{f(T)}\le C \max_{\overline\D}|f|,
\]
where $C$ is a constant depending only on $T$ (in fact, $C=\norm(\Phi)\norm(\Phi^{-1})$).

Here we should mention the works by Olofsson (see \cite{Olo08}), where he relates
certain transfer functions associated with $n$-hypercontractions with Bergman inner functions,
which are crucial in the description of invariant subspaces of Bergman spaces.
His results were further generalized in the work by
\marginp{comprobar}
Ball and Bolotnikov
\cite{BallBolotnikov2013a}, \cite{BallBolotnikov2013b}.


\section{Operators in $\cC_\al$ whose characteristic function has a determinant}
\label{sec-Sp-class}

In what follows, $\mathfrak{S}_p$ ($0<p\le \infty$) will denote the Schatten-von Neumann class of operators.
\begin{lemma}
Let $T \in \mathcal{C}_\alpha$ for some admissible function
$\alpha$ and let $p \in [1, \infty]$. \quad
\begin{enumerate}
\item[\textnormal{\textbf{(i)}}] If $I - T^*T \in \mathfrak{S}_p$,
then $D^2 \in \mathfrak{S}_p$. \item[\textnormal{\textbf{(ii)}}]
If $D^2 \in \mathfrak{S}_p$ and $\tilde{\alpha}$ has no zeros in $\overline\D$,
then $I - T^*T \in \mathfrak{S}_p$.
\end{enumerate}
\end{lemma}

\begin{proof}
(i) It is immediate, since $D^2 = \alpha[T^{*},T] =
\tilde{\alpha}[T^*,T](I-T^*T)$.

(ii) Since $\tilde{\alpha}$ has no zeros in $\overline\D$,
$1/\tilde{\alpha} \in A_W$, and we obtain that
\[
(1/\tilde{\alpha}) [T^*,T] (D^2) = (1/\tilde{\alpha}) [T^*,T]
(\tilde{\alpha}[T^*,T] (I- T^{*}T)) = I- T^{*}T,
\]
which proves the result.
\end{proof}

\begin{lemma}
Let $T \in \mathcal{C}_\alpha$ for some admissible function
$\alpha$ and let $p \in [1, \infty]$. If $\sigma(T) \ne
\overline{\D}$, then $D_{\tilde{T}} \in \mathfrak{S}_p$ if and
only if $D_{\tilde{T}^*} \in \mathfrak{S}_p$.
\end{lemma}

\begin{proof}
In the case when $0\notin\sigma(T)$, $D_{\tilde{T}}$ is
\marginp{rehacerlo}
unitarily equivalent
to $D_{\tilde{T}^*}$,
see \cite[the proof of Theorem VIII.1.1]{NFBK10} or \cite[Lemma 9]{Ker90}; this
implies our assertion.
The general case follows from this one. Indeed, take
any $\la\in\D\setminus\sigma(T)$ and consider the
M\"obius self-map of $\D$, given
by $b_\la(z)=(z-\la)/(1-\bar \la z)$.
Let $\tilde T_\la$ be the contraction, defined by $T_\la=b_\la(\tilde{T})$.
Then the formula
\[
I-\tilde T^*_\la \tilde T_\la = W^*(I-\tilde T^* \tilde T)W \qquad \text{with }
W = (1-|\la|^2)^{1/2}(I-\bar \la \tilde T)^{-1}
\]
implies that $D_{\tilde{T}}\in \mathfrak{S}_p$ if and only if
$D_{\tilde T_\la}\in\mathfrak{S}_p$. Since $0\notin\sigma(\tilde T_\la)$, the general
case follows.
\end{proof}

We remark that the previous lemma applies to a more general situation when
$T$ is a power bounded operator with $\sigma(T) \ne \overline{\D}$ and
$D$ is its abstract defect operator. Then, by Lemma~\ref{lem-defect},
one gets a contraction $\wT$, similar to $T$, and so for any $p$,
$D\in \mathfrak{S}_p$ iff $D_{\tilde T}\in \mathfrak{S}_p$ iff
$D_{\tilde{T}^*} \in \mathfrak{S}_p$.

We recall the well-known fact
that the characteristic function of a contraction $S$
has the determinant whenever
$\sigma(S) \ne \overline{\D}$ and
$I-S^*S\in\mathfrak{S}_1$ (this is the so-called class of weak contractions).
It follows that $\Theta_*$ has a determinant whenever
$\sigma(T) \ne \overline{\D}$ and
$D\in\mathfrak{S}_2$.
Notice that $\det \Theta_*$ is an $H^\infty$ function such that $\norm{\det \Theta_*}_\infty\le 1$.
We obtain the following statement.

\begin{prop}
Suppose $\al$ is strongly admissible and $T\in \cC_\al$
is a completely nonunitary operator.
Suppose also that  $\sigma(T) \ne \overline{\D}$ and
$\al[T^*,T]\in\mathfrak{S}_1$.
Denote by $\sigma_p(T)$ the point spectrum of $T$.
Then the following assertions
are equivalent.

\begin{itemize}

\item[(i)]
$T$ is complete, that is,
$H=\operatorname{span}\{\ker (\la I-T)^k: \quad k\ge 1, \la\in \sigma_p(T)\}$;
%

\item[(ii)]
$T^*$ is complete;

\item[(iii)]
$\det \Theta_*(z)$ is a Blaschke product.
\end{itemize}
\end{prop}

This follows from the above observations and from an analogous fact for
completely nonunitary contractions, see \cite[p. 134]{NikEasyReadVol2}.

In a similar way, one can extend the results by Treil~\cite{Treil97},
Nikolski - Benamara~\cite{BenamaraNikol99},
Kupin \cite{Kupin01}, \cite{Kupin03} and
\marginp{Nikolski spectral synthesis}
 others to the setting of operators
in $\cC_\al$, where $\al$ is strongly admissible.


\section{A result on the inclusion of classes $\mathcal{C}_\alpha \subset \mathcal{C}_\tau$}
\label{sec-cont}

The goal of this section is to prove the following result on the
containment of classes.

Let $\al$ and $\tau$ be admissible functions and let $\gamma := \wt{\tau} / \wt{\al}$. Note that $\gamma$ is analytic on a neighbourhood of the origin.

\begin{thm}\label{T2 B 27 GL}
Let $\al, \tau$ and $\ga$ be as above.
\begin{enumerate}
    \item[\textnormal{\textbf{(i)}}] \enspace If $\mathcal{C}_\al \subset \mathcal{C}_\tau$, then $\ga \succ 0$.
    \item[\textnormal{\textbf{(ii)}}] \enspace If $\ga \in A_W$, then $\mathcal{C}_\al \subset \mathcal{C}_\tau$ if and only if $\ga \succ 0$.
    \item[\textnormal{\textbf{(iii)}}] \enspace If $\ga$ is not bounded in $\D$ (in particular, if $\ga$ has poles in $\D$), then $\mathcal{C}_\al \not \subset \mathcal{C}_\tau$.
\end{enumerate}
\end{thm}



The following lemma is in the spirit of Pringsheim's Theorem (see \cite{Hil12}).

\begin{lemma}\label{L6 B 27 GL}
Let $g$ be a meromorphic function in $\D$, analytic in the origin,
such that $g \succcurlyeq 0$.
If $g$ is bounded on $[0,1)$, then $g$ is a bounded analytic function on $\D$.
\end{lemma}

\begin{proof}
Since $g \succcurlyeq 0$, we have
\begin{equation}\label{eq 1 B 28 GL}
\abs{g(z)} \leq g(\abs{z})
\end{equation}
whenever the series  for $g(\abs{z})$ converges.
Let $r\in(0,1]$ be the radius of convergence of $g$.
\marginp{He cambiado la prueba: estaba mal.}
If $r < 1$, then $g$ has poles on the circle $\{ \abs{z} = r \}$,
and this contradicts the boundedness of $g$ on $[0,r)$. Hence $r=1$, and therefore by\eqref{eq 1 B 28 GL}, $g$ is bounded on $\D$.
\end{proof}


\begin{lemma}\label{Lemma 8}
Let $\{ e_n \}_{n=0}^{\infty}$ be an orthonormal basis in $H$ and
\marginp{Más general: $\exists C: \Lambda_{n+k}\le C \Lambda_n, \; n,k\ge 0$}
let $\{ \la_n \}_{n=1}^{\infty}$ be a sequence of positive
numbers which are eventually $1$. Define the sequence $\Lambda =
\{ \Lambda_n \}_{n=0}^{\infty}$ by $\Lambda_n := (\la_1 \cdots
\la_{n+1})^2$. Let $T$ be the weighted shift operator, given by
$T e_n = \la_{n+1} e_{n+1}$, and let $\al \in
A_W$. Then $T \in \mathcal{C}_\al$ if and only
if $\al(\nabla)\Lambda \succ 0$ (here $\nabla$ is the backward shift on
one-sided sequences $\{A_n\}_{n\ge 0}$).
\end{lemma}

\begin{proof}
Note that $\Lambda$ is a bounded sequence and
that $T$ is a power bounded operator. Hence $\al[T^*,T]$ and
$\al(\nabla)\Lambda$ are well defined, so we need to prove that
$\al[T^*,T] \geq 0$ if and only if $\al(\nabla)\Lambda \succ 0$.

It is immediate that
\begin{equation}\label{eq. 2}
\al[T^*,T] \geq 0 \quad \Longleftrightarrow \quad
\sum_{n=0}^{\infty} \al_n \norm{T^n h}^2 \geq 0 \quad (\forall h \in
H).
\end{equation}
Next we observe that
\begin{equation}\label{eq. 3}
\sum_{n=0}^{\infty} \al_n \norm{T^n h}^2 \geq 0 \quad (\forall h \in
H) \quad \Longleftrightarrow \quad \sum_{n=0}^{\infty} \al_n
\norm{T^n e_j}^2 \geq 0 \quad (\forall j \geq 0).
\end{equation}
Indeed, the direct implication is obvious. The converse is seen from the
following formula, valid for any vector
$h =
\sum_{n=0}^{\infty} h_n e_n \in H$, where $\{ h_n \}_{n=0}^{\infty}
\in \ell^2$:
\[
\sum_{n=0}^{\infty} \al_n \norm{T^n h}^2 =
\sum_{n=0}^{\infty} \al_n \bigg( \sum_{j=0}^{\infty} \abs{h_j}^2 \norm{T^n e_j}^2 \bigg) \\
= \sum_{j=0}^{\infty} \abs{h_j}^2 \bigg( \sum_{n=0}^{\infty} \al_n \norm{T^n e_j}^2  \bigg)
\]
(we can change the order of summation because the series converge absolutely).

%

Fix $j \geq 0$. For every $n \geq 0$ we have
\[
T^n e_j = \la_{j+1} \la_{j+2} \cdots \la_{j+n} e_{j+n} =
\sqrt{\dfrac{\Lambda_{n+j}}{\Lambda_j}} \,  e_{j+n},
\]
so $\norm{T^n e_j}^2 = \Lambda_{n+j} / \Lambda_j$ and therefore
\[
\sum_{n=0}^{\infty} \al_n \norm{T^n e_j}^2 = \sum_{n=0}^{\infty} \al_n
\frac{\Lambda_{n+j}}{\Lambda_j} = \frac{1}{\Lambda_j}
[\al(\nabla)\Lambda]_j.
\]
Hence it follows that
\begin{equation}\label{eq. 4}
\sum_{n=0}^{\infty} \al_n \norm{T^n e_j}^2 \geq 0 \quad (\forall j
\geq 0) \quad \Longleftrightarrow \quad \al(\nabla)\Lambda \succ 0.
\end{equation}
The statement now follows from \eqref{eq. 2}, \eqref{eq. 3} and \eqref{eq. 4}.
\end{proof}

Before starting the proof of Theorem \ref{T2 B 27 GL}, let us make a
final observation.

\begin{rem}\label{Remark 2}
If $T \in L(H)$ is a power bounded operator and $\al \in
A_W$ then
\[
\begin{split}
    \al[T^*,T] \geq 0 \quad &\Longleftrightarrow \quad \sum_{n=0}^{\infty} \al_n \norm{T^n  h}^2 \geq 0 \quad (\forall h \in H) \\
    &\Longleftrightarrow \quad \sum_{n=0}^{\infty} \al_n \norm{T^{n+j} h}^2 \geq 0 \quad (\forall j \geq 0, \forall h \in H),
\end{split}
\]
where in the last equivalence we just change $h$ by $T^j h$.
Therefore, if we fix $h \in H$ and define the sequence $\Lambda =
\{ \Lambda_n \}_{n=0}^{\infty}$ by $\Lambda_n := \norm{T^n h}^2$
then $\al[T^*,T] \geq 0$ implies that $\al(\nabla)\Lambda \succ 0$.
\end{rem}

\begin{proof}[Proof of Theorem \ref{T2 B 27 GL}]


(i) Suppose that $\gamma \not \succ 0$ and let $\ell$ be the smallest index
such that $\gamma_\ell < 0$. Note that $\ell \geq 1$ because
$\gamma_0 >0$. By Lemma \ref{Lemma 8}, we just need to find a
sequence of positive numbers $\Lambda = \{ \Lambda_n
\}_{n=0}^{\infty}$ that is eventually constant such that
$\alpha(\nabla)\Lambda \succ 0$ and $\tau(\Lambda) \not \succ 0$,
because in that case if we fix an orthonormal basis $\{ e_n
\}_{n=0}^{\infty}$ of $H$ then the weighted shift operator $T$
defined by $Te_n = \sqrt{\Lambda_{n+1}/\Lambda_n}\, e_{n+1}$
satisfies $T \in \mathcal{C}_\alpha \setminus \mathcal{C}_\tau$.
Let us construct that sequence.

Consider the sequence $\Gamma := (\gamma_\ell, \gamma_{\ell
- 1}, \ldots, \gamma_0, 0, 0 , \ldots)$ and define the sequence
$\Psi$ by
\[
\Psi := (\tilde{\tau})^{-1}(\nabla) \Gamma.
\]
Note that $\Psi$ is well defined because $\Gamma$ has only
finitely many nonzero terms and $\tilde{\tau}$ is invertible in a
neighbourhood of the origin.

Since $\Gamma_n = 0$ for $n \geq \ell +1$ we obtain that also
$\Psi_n = 0$ for $n \geq \ell +1$. Finally, let $\Lambda$ be a the
sequence that satisfies
\[
\Psi = (1-\nabla) \Lambda,
\]
with $\Lambda_0$ large enough so that $\Lambda_n > 0$ for every
$n$. Note that $\Lambda_n$ is constant for $n \geq \ell + 1$ and
it also satisfies
\[
\alpha(\nabla) \Lambda = \tilde{\alpha} (\nabla) \Psi = \left(
\dfrac{\tilde{\alpha}}{\tilde{\tau}} \right) (\nabla) \Gamma =
\left( \dfrac{1}{\gamma} \right) (\nabla) \Gamma = (0, \ldots,
0,\overset{\substack{\ell\\\smile}}{1},0,\ldots) \succ 0
\]
and
\[
\tau(\nabla) \Lambda = \tilde{\tau}(\nabla) \Psi = \Gamma \not
\succ 0,
\]
since $\Gamma_0 = \gamma_\ell < 0$. Hence (i) is proved.

(ii) It is clear that $\cC_\al \subset \cC_\tau$ if $\ga \succ 0$. Indeed, if $T \in \cC_\al$, then using Lemma \ref{L3 B 27 GL} (ii) we have
\[
\tau[T^*,T] = (\ga \al)[T^*,T] = \sum_{n=0}^{\infty} \ga_n T^{*n} \al[T^*,T] T^{n} \geq 0,
\]
hence $T \in \cC_\tau$. The other implication follows from (i).

(iii) Note that Lemma \ref{L6 B 27 GL} implies that if $\ga$ is not bounded on
$\D$ then $\ga \not \succ 0$ (recall that $\ga$ has no poles in $[0,1]$), so this statement also follows from (i) and the theorem is proved.
\end{proof}

\begin{rem}
Note that in general, for a rational admissible function $r$, it
is not possible to find an admissible polynomial $p$ such that
$\mathcal{C}_r \subset \mathcal{C}_p$. For example, consider the
rational admissible function
\[
r(t) = \dfrac{1-t}{1-t/2}.
\]
If such a $p$ exists, say $p(t) = (1-t) \tilde{p}(t)$, then by
Theorem \ref{T2 B 27 GL} (i) we should have $\tilde{p}(t)(1-t/2) \succ
0$. But it is immediate to check that this is impossible for any
real polynomial $\tilde{p}$.
\end{rem}

\begin{rem}
Let $\al$ be an admissible function. Since $\wt{\al} \in A_W$, by Theorem \ref{T2 B 27 GL} (i)
we have that $\cC_{1-t} \ss \cC_\al$ if and only if $\wt{\al} \succ 0$. If moreover $\wt{\al}$ has zeroes in $\ol{\D}$, then
by Theorem \ref{T2 B 27 GL} (iii) we deduce that $\cC_{1-t} \subsetneqq \cC_\al$. For example
\[
\cC_{1-t} \subsetneqq \cC_{(1-t)(t^2 + 1/4)}.
\]
\end{rem}

\begin{rem}
Consider the strongly admissible rational function
\[
\alpha(t) = \frac{1-t}{\frac{4}{9} (t-\frac{3}{2})^2}.
\]
Then $\mathcal{C}_\alpha \not \subset \mathcal{C}_{1-t^n}$ for
every $n\geq 1$. Indeed, suppose on the contrary that
$\mathcal{C}_\alpha \subset \mathcal{C}_{1-t^n}$ for some $n\geq
1$. Then, by Theorem \ref{T2 B 27 GL} (i), we should have
\[
\frac{4}{9} \left(t-\frac{3}{2}\right)^2 (1+t+\cdots+t^{n-1})
\succ 0;
\]
however, the coefficient of $t$ in the above expression is $-1/3$. This shows that
there are classes $\cC_\al$ that are not contained in the union of classes
$\cC_{1-t^n}$ over all $n\ge 1$.
\end{rem}

\begin{lemma} \label{C_a union C_b contenido en C_c}
Let $\al$ and $\be$ be admissible functions, and consider the following conditions.

\begin{enumerate}
    \item[(a)] $\al, \be \in \mathcal{H}(\ol{\D})$;
    \item[(b)]  $\al(t)/(1-t)$ or $\be(t)/(1-t)$ has no zeros on $\overline\D$.
\end{enumerate}
If (a) or (b) holds, then there
exists an admissible function $\ga$ such that $\cC_\al \cup \cC_\be \ss \cC_\ga$.

\begin{proof}
(a) Suppose that $\al, \be \in \mathcal{H}(\ol{\D})$. Then
\[
\frac{\al}{\be} = \frac{a}{b}\, \vp
\]
for some polynomials $a$ and $b$ without common roots and a function $\vp \in \mathcal{H}(\ol{\D})$ which is positive on $[0,1]$. By Lemma \ref{L2 B 27 GL}, there exists a function $\psi \in A_W$ such that $\psi \succ 0$ and $\psi \vp \succ 0$. Put $a = a_+ a_- a_{nr}$ where $a_+$ contains the positive roots of $a$, $a_-$ contains the negative roots of $a$ and $a_{nr}$ contains the non-real roots of $a$. If for example $a$ does not have any positive root, then we just put $a_- = 1$. In the same way, put $b = b_+ b_- b_{nr}$.
Applying Corollary \ref{C1 B 27 GL} twice, notice that there exists a polynomial $p$ without roots in $\overline{\D}$ such that $p \succ 0$, $pa_{nr} \succ 0$ and $pb_{nr} \succ 0$. Let
\[
v:= \frac{a_- a_{nr} p}{b_+}, \quad w:= \frac{b_- b_{nr} p}{a_+}.
\]
Note that $v \succ 0, w \succ 0$ and $a/b = v/w$. Now we simply put $\ga := a w \psi = b v \psi \vp$. Since $w \psi \succ 0$ and $v \psi \vp \succ 0$, the result follows from Theorem \ref{T2 B 27 GL} (i).

(b) Suppose that
$\be(t)/(1-t)$ has no zeros on $\overline\D$. Then $\al/\be =: \vp
\in A_W$ is positive on $[0,1]$. By Lemma \ref{L2 B 27 GL}, there
exists a function $\psi \in A_W$ such that $\psi \succ 0$ and
$\psi \vp \succ 0$. Now we put $\ga := \al \psi = \be \vp \psi$.
Since $\psi \succ 0$ and $\vp \psi\succ 0$, the result follows
from Theorem \ref{T2 B 27 GL} (ii).
\end{proof}
\end{lemma}

\begin{cor}
\label{cor orth sum matr}
Let $T_1$ be a complex square matrix and
$T_2$ be a Hilbert space operator. If $T_1$ has no
Jordan blocks and $\sigma(T_1) \ss \T$, whereas
the spectral radius of $T_2$ is less than $1$, then there exists
an admissible function $\ga$ such that $T_1 \oplus T_2 \in \cC_\ga$.
\end{cor}

\begin{proof}
This follows immediately from
Lemma \ref{matrix without Jordan blocks}, Proposition \ref{properties of class C_a} (b) and (e), and Lemma~\ref{C_a union C_b contenido en C_c}.
\end{proof}


\section{Existence of the limit of $\norm{T^n h}^2$}
\label{sec-existence-limit}

As it was explained at the end of Section~\ref{sec-abstr-def-oper},
in general, the limit of norms $\norm{T^n h}$
as $n\to\infty$ (where $h\in H$) does not exist.
In this section we prove the following result:

\begin{thm}\label{Theorem Extra 1}
Let $\alpha(t) = (1-t) \tilde{\alpha}(t)$, where
$\alpha$ is strongly admissible,
and let $T \in
\mathcal{C}_\alpha$. Then, for every $h \in H$, there exists
the limit $\lim_{n \to \infty} \norm{T^n h}^2$.
\end{thm}

We will use the backward shift  $\nabla$ and
the shift $\nabla_-$, acting on one-sided bounded sequences $\{ a_n \}_{n=0}^{\infty}$.
They are given by $[\nabla a]_n = a_{n+1}$ for every $n \geq 0$ and
$[\nabla_- a]_0 = 0$, $[\nabla_-a]_n = a_{n-1}$ for every $n \geq
1$.

%
%
%
%

If we identify the sequence $a = \{ a_n \}_{n=0}^{\infty}$ with
the power series $a(z) = \sum_{n=0}^{\infty} a_n z^n$, then we can
identify the operator $\nabla$ and $\nabla_-$ with the operators
given by
\[
(\nabla a)(z) = \frac{a(z)-a(0)}{z}, \quad (\nabla_- a)(z) =
za(z).
\]
It is clear that $\nabla \nabla_- = I$ and $\nabla_-\nabla a (z) =
a(z)-a(0)$.

Given a function $f \in A_W$,
the operators $f(\nabla)$ and
$f(\nabla_-)$ are well-defined.
Note that $c=f(\nabla_-)a$ is given by
$c_n := \sum_{j=0}^{n} f_j a_{n-j}$. In terms of
power series, one just has $c(z) = f(z)a(z)$.
We can say, in fact, that in the power
series representation, $f(\nabla_-)$ is an analytic Toeplitz
operator and $f(\nabla)$ is an anti-analytic Toeplitz operator.

The following formula will be useful:
\begin{equation}\label{eq 1 Extra 1}
\nabla_{-}^{k} \nabla^j a(z) = z^{k-j}(a(z) - a_{j-1}z^{j-1} -
\cdots - a_1 z - a_0).
\end{equation}
We need some auxiliary lemmas.

\begin{lemma}\label{Lemma 1 Extra 1}
Let $f,g \in A_W$ and let $a \in \ell^{\infty}$.
Then
\[
f(\nabla)[g(\nabla)a] = (fg)(\nabla)a.
\]
\end{lemma}
The proof is immediate just doing a change of summation indices.

\begin{lemma}\label{Lemma 2 Extra 1}
Let $f \in A_W$ and let $a \in \ell^{\infty}$ be a
convergent sequence, say $a_n \to a_\infty$.
\begin{enumerate}
\item[\textnormal{\textbf{(i)}}] \enspace If $b = f(\nabla) a$,
then $b_n \to f(1)a_\infty$.

\item[\textnormal{\textbf{(ii)}}] \enspace The same is true for
$\nabla_-$ in place of $\nabla$. Namely, if $c = f(\nabla_-) a$,
then also $c_n \to f(1)a_\infty$.
\end{enumerate}
\end{lemma}

\begin{proof}
Both statements are straightforward, and we will only check (i).
Fix $\varepsilon > 0$ and let $\abs{a_n - a_\infty}< \varepsilon /
\sum_{j=0}^{\infty} \abs{f_j}$ for every $n \geq N$. Then
\[
\abs{b_n - f(1)a_\infty} \leq \sum_{j=0}^{\infty} \abs{f_j}
\abs{a_{n+j}-a_\infty} < \varepsilon
\]
for every $n \geq N$.
\end{proof}

%
%

We can rephrase part (ii) of last lemma in terms of formal power
series as follows.

\begin{cor}\label{Corollary 1 Extra 1}
Let $f \in A_W$ and let $a(z) = \sum_{n=0}^{\infty}
a_n z^n$ be a formal power series where the sequence $\{ a_n
\}_{n=0}^{\infty}$ converges to some number $a_\infty \in \R$. If
$b(z) = f(z)a(z)$, then $b_n \to f(1)a_\infty$.
\end{cor}

\begin{lemma}\label{Lemma 4 Extra 1}
Let $q$ be a real polynomial whose roots are in $\D$. Let $a \in
\ell^{\infty}$ and put $b = q(\nabla)a$. If $b_n \to b_\infty \in
\R$, then $a_n \to b_\infty / q(1)$.
\end{lemma}

\begin{proof}
Put $q(t) = q_s t^s + \cdots + q_1 t + q_0$. Then
\[
\nabla_{-}^{s} b = (q_0 \nabla_{-}^{s} + q_1 \nabla_{-}^{s} \nabla
+ \cdots + q_s \nabla_{-}^{s} \nabla^s)a,
\]
which can be written in formal power series using \eqref{eq 1
Extra 1} as
\[
z^s b(z) = q_0 z^s a(z) + q_1 z^{s-1} (a(z)-a_0) + \cdots + q_s
(a(z) - a_{s-1}z^{s-1} - \cdots -a_1 z - a_0).
\]
So if we put $\tilde{q}(t) = q_0 t^s + q_1 t^{s-1} + \cdots +
q_s$, then
\[
z^s b(z) = \tilde{q}(z) a(z) - r(z)
\]
for some polynomial $r$ of degree at most $s-1$. Note that
$\tilde{q}$ has no roots in $\overline{\D}$, hence $1/\tilde{q}
\in A_W$ and therefore
\[
a(z) = \frac{z^s}{\tilde{q}(z)}\, b(z) + \frac{r(z)}{\tilde{q}(z)}.
\]
Since $r/\tilde{q} \in A_W$, its $n$-th Taylor
coefficient tends to $0$. Now the statement follows using the
previous corollary and that $\tilde{q}(1) = q(1)$.
\end{proof}

\begin{lemma}\label{Lemma 5 Extra 1}
Let $q$ be a real polynomial whose roots are in $\D$ and put $Q(t)
= (1-t)q(t)$. Let $a \in \ell^{\infty}$. If $b = Q(\nabla) a$ and
$b_n \geq 0$, then there exists $\lim a_n$.
\end{lemma}

\begin{proof}
Put $c:= q(\nabla)a$. Then $b = q(\nabla)a - \nabla q(\nabla)a$,
so $b_n = c_n - c_{n+1} \geq 0$. Hence $\{ c_n \}_{n=0}^{\infty}$
is a decreasing sequence. Since $\norm{c}_\infty \leq (\abs{q_0} +
\cdots + \abs{q_s}) \norm{a}_\infty$, the sequence $c$ is bounded.
Therefore $\{ c_n \}_{n=0}^{\infty}$ converges and by the previous
lemma we obtain that $a_n \to c_\infty / q(1)$.
\end{proof}

\begin{proof}[Proof of Theorem \ref{Theorem Extra 1}]
Let $h \in H$. Since $T \in \mathcal{C}_\alpha$, we obtain that $\sum_{n=0}^{\infty} \alpha_n \norm{T^n h}^2 \geq 0$. Changing $h$
by $T^j h$ for $j \geq 1$ we get that $\sum_{n=0}^{\infty}
\alpha_n \norm{T^{n+j} h}^2 \geq 0$. Hence, if we define the
sequence $a$ by $a_n = \norm{T^n h}^2$ then we have that $b :=
\alpha(\nabla)a$ satisfies $b_n \geq 0$ for every $n \geq 0$.

By Lemma \ref{Lemma 1 Extra 1} we have $b = \alpha(\nabla) a =
(1-\nabla) \tilde{\alpha}(\nabla) a$. Since $\tilde{\alpha}$ does
not vanish on $\T$, we can split it as
\[
\tilde{\alpha} = q \tilde{\beta},
\]
where $q$ is a polynomial with roots in $\D$ and $\tilde{\beta}
\in A_W$ does not vanish on $\overline{\D}$.
Therefore, if we put $Q(t) = (1-t)q(t)$ and $c :=
\tilde{\beta}(\nabla)a$, then $b = Q(\nabla) c$. By Lemma
\ref{Lemma 5 Extra 1} we know that there exists $\lim c_n$, and
since $a = (1/\tilde{\beta})(\nabla)c$ and $(1/\tilde{\beta}) \in
A_W$, the statement follows by Corollary
\ref{Corollary 1 Extra 1}.
\end{proof}


\begin{cor}[of Theorem~\ref{Theorem Extra 1}]
Assume the hypotheses of Theorem \ref{thm new norm}, in particular,
that $T\in\cC_\al$ and that $f$, $B$ are defined as in this theorem.
If $\al$ is strongly admissible, then the norm defined in \eqref{eq. new norm} can
be alternatively expressed by
\[
\Norm{h}^2 = \sum_{n=0}^{\infty} \norm{DT^n h}^2  + \lim_{n \to \infty} \norm{T^nh}^2.
\]
\end{cor}

This follows from Theorem~\ref{Theorem Extra 1} and formula~\eqref{eq-lim-st}.
\bibliographystyle{abbrv}

\bibliography{biblio_Memoria_Dm}

\end{document}